\documentclass[11pt, a4paper, oneside]{amsart}

\usepackage{pdfsync}
\usepackage[english]{babel}
\usepackage{amsmath, amsthm, amsfonts, mathrsfs, amssymb}
\usepackage{mathtools}
\mathtoolsset{centercolon}
\usepackage{booktabs}
\usepackage[shortlabels]{enumitem}
\setlist[itemize]{leftmargin=20pt}
\usepackage[hmargin=3cm,vmargin=3cm]{geometry}
\usepackage[color=green!40]{todonotes}
\usepackage{bm}
\usepackage{bbm}

\usepackage{color}
\usepackage{graphicx}
\usepackage{tikz-cd}
\usetikzlibrary{arrows}

\usepackage[pdftex,bookmarks,colorlinks,breaklinks]{hyperref}  
\definecolor{dullmagenta}{rgb}{0.4,0,0.4}   
\definecolor{darkblue}{rgb}{0,0,0.4}
\definecolor{darkgreen}{rgb}{0,0.4,0}
\hypersetup{linkcolor=darkblue,citecolor=blue,filecolor=dullmagenta,urlcolor=darkblue} 

\newtheorem{TheoremLetter}{Theorem}
{}

\newcommand{\N}{\mathbb{N}}
\newcommand{\R}{\mathbb{R}}
\newcommand{\C}{\mathbb{R}}

\newcommand{\D}{\mathscr{D}}

\newcommand{\Sp}{\mathcal{S}}
\newcommand{\A}{\mathcal{A}}

\newcommand{\eps}{\varepsilon}

\DeclareMathOperator{\supp}{supp}

\DeclareMathOperator*{\essinf}{ess\,inf}

\DeclareMathOperator{\ind}{\mathbbm{1}}
\DeclareMathOperator{\wk}{weak}

\newcommand{\loc}{\operatorname{loc}}
\newcommand{\BMO}{\operatorname{BMO}}
\newcommand{\Dini}{\operatorname{Dini}}

\newcommand{\dd}{\,\mathrm{d}}

\def\avint_#1{\mathchoice{\mathop{\kern 0.2em\vrule width 0.6em height 0.69678ex depth -0.58065ex \kern -0.8em \intop}\nolimits_{\kern -0.4em#1}}{\mathop{\kern 0.1em\vrule width 0.5em height 0.69678ex depth -0.60387ex \kern -0.6em \intop}\nolimits_{#1}} {\mathop{\kern 0.1em\vrule width 0.5em height 0.69678ex depth -0.60387ex \kern -0.6em \intop}\nolimits_{#1}} {\mathop{\kern 0.1em\vrule width 0.5em height 0.69678ex depth -0.60387ex \kern -0.6em \intop}\nolimits_{#1}}}

\newtheorem{theorem}{Theorem}

\newtheorem{lemma}[theorem]{Lemma}
\newtheorem{proposition}[theorem]{Proposition}

\theoremstyle{remark}
\newtheorem{remark}[theorem]{Remark}
\newtheorem*{remark*}{Remark}

\theoremstyle{definition}
\newtheorem{definition}[theorem]{Definition}

\numberwithin{theorem}{section}
\numberwithin{equation}{section}
\title[Weighted BMO estimates and extrapolation]{Weighted BMO estimates for singular integrals and endpoint extrapolation in Banach function spaces}
\author{Zoe Nieraeth and Guillermo Rey}

\address[Zoe Nieraeth]{BCAM\textendash  Basque Center for Applied Mathematics, Bilbao, Spain}
\email{zoe.nieraeth@gmail.com}

\address[Guillermo Rey]{Universidad Aut\'onoma de Madrid}
\email{guillermo.rey@uam.es}

\allowdisplaybreaks

\begin{document}
\begin{abstract}
  In this paper we prove sharp weighted BMO estimates for singular integrals, and we show how such estimates can be extrapolated to Banach function spaces.
\end{abstract}

\keywords{Sparse operators, Muckenhoupt weights, Banach function spaces, Calder\'on-Zygmund operators, BMO, Rubio de Francia extrapolation}

\subjclass[2010]{42B20, 46E30}


\maketitle

\section{Introduction}

In the 70's Muckenhoupt and Wheeden extended the known $L^\infty \to \BMO$ estimates for the Hilbert transform to the weighted setting. In particular,
they proved that
\begin{align*}
  \int_I |Hf - \langle Hf \rangle_I| \dd x \lesssim_w \|fw\|_{L^\infty} \int_I w^{-1} \dd x,
\end{align*}
where $\langle Hf \rangle_I$ is the average of $Hf$ over $I$, holds if and only if $w^{-1} \in A_\infty \cap B_2$. This is Theorem 1 in \cite{MW_Hilbert}.

The class $B_2$, which was introduced in \cite{HMW_Hilbert}, is the collection of all locally integrable weights $w$ satisfying
\begin{align*}
  \int_{I^c} \frac{|I|w(x)}{|x-c_I|^2} \dd x \lesssim \langle w \rangle_I.
\end{align*}

In the same paper (Theorem 2), Muckenhoupt and Wheeden characterize $w^{-1} \in A_1$ as the class of weights $w$
for which the following weighted BMO estimate holds
\begin{align} \label{intro:bmo2}
  \|w\|_{L^\infty(I)} \frac{1}{|I|}\int_I |Hf-\langle Hf\rangle_I| \dd x \lesssim_w \|fw\|_{L^\infty}
\end{align}
for all intervals $I$ and all functions $f$.

This result was later used by Harboure, Mac\'ias, and Segovia in \cite{HMS} to prove a $\BMO$-to-$L^p$ extrapolation result, saying that if a sublinear
operator $T$ satisfies the same bound as in \eqref{intro:bmo2} with implicit constant depending only on $[w^{-1}]_{A_1}$, then
\begin{align*}
  \|Tf\|_{L^p(v)} \lesssim_v \|f\|_{L^p(v)}
\end{align*}
for all weights $v \in A_p$ and $1 < p < \infty$.

We should mention the result of \cite{Zoe}, where a multilinear extrapolation result is obtained which generalizes and sharpens the one by Harboure, Mac\'ias, and Segovia.
See also \cite{CPR}, where sharp linear versions of these two results are obtained.

The purpose of this article is twofold: first, we extend the extrapolation result of Harboure, Mac\'ias, and Segovia, to the setting of Banach function spaces.
Second, we extend the results of Muckenhoupt and Wheeden to sparse operators as well as to Calder\'on-Zygmund operators
whose kernels satisfy a Dini smoothness condition. We fully recover the results of \cite{CPR}, and in fact our extrapolation result is sharper, see Remark \ref{rem:sharpextrap}.

Let us first state a simplified version of our extrapolation theorem.
\begin{TheoremLetter}
  Let $T$ be a sublinear operator.
  Suppose there is an increasing function $\phi:[1,\infty)\to(0,\infty)$ such that for all weights $w$ with $w^{-1}\in A_1$,
  and all compactly supported functions $f$ with $fw\in L^\infty(\R^d)$ we have
  \begin{align*}
    \|w\|_{L^\infty(Q)} \frac{1}{|Q|}\int_Q |Tf - \langle Tf \rangle_Q| \dd x \leq\phi([w^{-1}]_{A_1})\|fw\|_{L^\infty(\R^d)}.
  \end{align*}
  Then, for all Banach function spaces $X$ over $\R^d$ for which $M$ is bounded on $X$ and its associate space $X'$, and all compactly supported functions $f\in X$
  \begin{align*}
    \|Tf\|_{X}\lesssim_d\|M\|_{X'\to X'}\phi(2\|M\|_{X\to X})\|f\|_X.
  \end{align*}
\end{TheoremLetter}
Our full result is Theorem \ref{thm:sharpbmoextrap} in section \ref{section:extrapolation}.

As for our result for Calder\'on-Zygmund operators, suppose $T$ is an operator represented by
\begin{align*}
  Tf(x) = \int_{\R^d} K(x,y)f(y) \dd y
\end{align*}
for all $x$ outside of the support of $f$, and where $K$ satisfies
\begin{align*}
  |K(x,y) - K(z,y)| \leq \Omega\biggl( \frac{|x-z|}{|x-y|} \biggr) \frac{1}{|x-y|^d}
\end{align*}
for all $x,y,z$ satisfying $|x-y| > 2|x-z| > 0$, and where $\Omega: [0,\infty) \to [0, \infty)$ is an increasing subadditive function with $\Omega(0) =0$.

Similar to the situation in \cite{MW_Hilbert}, it is not sufficient that $w \in A_\infty$ for $T$ to have weighted BMO estimates.
What we need is a generalization of the $B_2$ condition that takes into account the interaction between the weight and the smoothness of the kernel.
This is quantified by
\begin{align*}
  [w]_{B(\Omega)} = \sup_Q \frac{|Q|}{w(Q)} \int_{Q^c}\frac{w(x)}{|x-c_Q|^d}\Omega\Bigl( \frac{\ell(Q)}{|x-c_Q|}\Bigr) \dd x.
\end{align*}
Note that, in dimension one and when $\Omega(t) = t$, $[w]_{B(\Omega)} < \infty$ is the $B_2$ condition.

With this definition, we have
\begin{TheoremLetter}
  Let $w^{-1} \in L^1_{\loc}(\R^d)$ and let $T$ be an operator as above, then
  \begin{align*}
    \frac{1}{|Q|}\int_Q |Tf - \langle Tf \rangle_Q| \dd x &\lesssim \bigl([1]_{B(\Omega)} + [w^{-1}]_{B(\Omega)}\bigr)[w^{-1}]_{A_\infty} \langle w^{-1} \rangle_Q \|fw\|_{L^\infty(\R^d)}
  \end{align*}
  and
  \begin{align*}
    \|w\|_{L^\infty(Q)}\frac{1}{|Q|}\int_Q |Tf - \langle Tf \rangle_Q| \dd x &\lesssim [w^{-1}]_{A_1}^2[w^{-1}]_{A_\infty}\|\Omega\|_{\Dini}  \|fw\|_{L^\infty(\R^d)}
  \end{align*}
  for all compactly supported functions $f$ with $fw \in L^\infty(\R^d)$.
\end{TheoremLetter}
This is Theorem \ref{czo:theorem} in section \ref{section:czo}.
These estimates are based on sparse domination techniques, so they offer a simplified argument to the results from \cite{MW_Hilbert} and \cite{MW_H1}.

\subsection*{Notation}
We work in the setting of $\R^d$ equipped with the Lebesgue measure $\mathrm{d}x$. For a measurable set $E\subseteq\R^d$ we denote its measure by $|E|$. For a measurable function $f\in L^0(\R^d)$, a measurable set $E\subseteq\R^d$ of finite positive measure, and a $q\in(0,\infty)$ we write
\[
\langle f\rangle_{q,E}:=\left(\frac{1}{|E|}\int_E\!|f|^q\,\mathrm{d}x\right)^{\frac{1}{q}},
\]
and denote the essential supremum of $f$ in $E$ by $\langle f\rangle_{\infty,E}$. Moreover, we denote the linearized mean of $f$ on $E$ by 
\[
\langle f\rangle_E:=\frac{1}{|E|}\int_E\!f\,\mathrm{d}x.
\]

By a \emph{cube} in $\R^d$ we mean a half-open cube whose sides are parallel to the coordinate axes. 

We will write $A\lesssim_{a,b,\ldots} B$ to mean that there is a constant $C$ depending only on the parameters $a,b,\cdots$, such that $A\leq CB$. We write $A\eqsim_{a,b,\cdots} B$ when both $A\lesssim_{a,b,\cdots} B$ and $B\lesssim_{a,b,\cdots} A$.

\section{Preliminaries}\label{sec:prelim}
\subsection{Dyadic analysis}
In this paper we will be working in $\R^d$, but our results are also valid in more general spaces of homogeneous type.
We will often reduce our arguments to dyadic grids: a dyadic grid is a collection of cubes with the property
\begin{align*}
  P \cap Q \neq \emptyset \implies P \subseteq Q \text{ or } Q \subseteq P.
\end{align*}

For a collection of cubes $\mathscr{E}$ in a dyadic grid $\D$ and a cube $Q_0\in\D$, we let $\mathscr{E}(Q_0)$ denote the cubes in $\mathscr{E}$ that are contained in $Q_0$.
We refer the interested reader to the monograph \cite{LNBook}.

\subsection{Muckenhoupt weights}\label{sec:weights}

A weight $w$ in $\R^d$ can be associated with the measure through $w(E):=\int_E\!w\,\mathrm{d}x$. A classical result by Muckenhoupt is that the Hardy-Littlewood maximal operator
\[
Mf:=\sup_Q\langle f\rangle_{1,Q}\ind_Q,
\]
where the supremum is taken over all cubes $Q$ in $\R^d$, is bounded $L^p(\R^d,w)\to L^p(\R^d,w)$ for $p\in(1,\infty)$ precisely when $w$ satisfies the $A_p$ condition
\[
[w]_{A_p}:=\sup_Q\langle w\rangle_{1,Q}\langle w^{-\frac{1}{p}}\rangle_{p',Q}^p<\infty,
\]
and bounded $L^1(\R^d,w)\to L^{1,\infty}(\R^d,w)$ if and only if $w$ satisfies the $A_1$ condition
\[
[w]_{A_1}:=\sup_Q\langle w\rangle_{1,Q}\langle w^{-1}\rangle_{\infty,Q}=\|(Mw)w^{-1}\|_{L^\infty(\R^d)}<\infty.
\]
However, in the case $p=\infty$, we have $L^\infty(\R^d,w)=L^\infty(\R^d)$ for all weights $w$ and hence, $M$ is bounded $L^\infty(\R^d,w)\to L^\infty(\R^d,w)$ for all weights $w$.

Instead, if we define $L^p_w(\R^d)$ through $\|f\|_{L^p_w(\R^d)}:=\|fw\|_{L^p(\R^d)}$, then we do get a meaningful condition for $p=\infty$. Note that for $p\in(1,\infty)$ we have $L^p_w(\R^d)=L^p(\R^d:w^p)$ and hence, $M:L^p_w(\R^d)\to L^p_w(\R^d)$ is bounded precisely when
\[
[w^p]_{A_p}^{\frac{1}{p}}=\sup_Q\langle w\rangle_{p,Q}\langle w^{-1}\rangle_{p',Q}<\infty.
\]
For $p=\infty$, this expression yields the condition
\[
\sup_Q\langle w\rangle_{\infty,Q}\langle w^{-1}\rangle_{1,Q}=\|M(w^{-1})w\|_{L^\infty(\R^d)}=[w^{-1}]_{A_1}.
\]
Indeed, defining $L^\infty_w(\R^d)$ through $\|f\|_{L^\infty_w(\R^d)}=\|fw\|_{L^\infty(\R^d)}$, we have the following result:
\begin{proposition}\label{prop:winva1char}
Let $w$ be a weight in $\R^d$. Then $M$ is bounded $L^\infty_w(\R^d)\to L^\infty_w(\R^d)$ if and only if $w^{-1}\in A_1$. In this case we have
\[
\|M\|_{L^\infty_w(\R^d)\to L^\infty_w(\R^d)}=[w^{-1}]_{A_1}.
\]
\end{proposition}
\begin{proof}
If $M$ is bounded $L^\infty_w(\R^d)\to L^\infty_w(\R^d)$, then since $f=w^{-1}$ satisfies $\|f\|_{L^\infty_w(\R^d)}=1$, we have
\[
[w^{-1}]_{A_1}=\|M(w^{-1})\|_{L_w^\infty(\R^d)}\leq\|M\|_{L^\infty_w(\R^d)\to L^\infty_w(\R^d)}.
\]
For the converse, if $w^{-1}\in A_1$, then for all $f\in L^\infty_w(\R^d)$ we have
\[
\|Mf\|_{L^\infty_w(\R^d)}\leq\|f\|_{L^\infty_w(\R^d)}\|M(w^{-1})\|_{L^\infty_w(\R^d)}=[w^{-1}]_{A_1}\|f\|_{L^\infty_w(\R^d)}.
\]
The assertion follows.
\end{proof}

If $\D$ is a dyadic grid in $\R^d$, then we define
\[
M^\D f:=\sup_{Q\in\D}\langle f\rangle_{1,Q}\ind_Q
\]
and the class $A_1(\D)$ through
\[
[w]_{A_1(\D)}:=\sup_{Q\in\D}\langle w\rangle_{1,Q}\langle w^{-1}\rangle_{\infty,Q}=\|(M^\D w)w^{-1}\|_{L^\infty(\R^d)}.
\]
Then, exactly as in the above result, we have $\|M^\D\|_{L^\infty_w(\R^d)\to L^\infty_w(\R^d)}=[w^{-1}]_{A_1(\D)}$.

For a cube $Q\in\D$ we denote by $\D(Q)$ the collection of cubes in $\D$ that are contained in $Q$. Defining $M^{\D(Q)}$ accordingly, we define the Fujii-Wilson $A_\infty$ constant of a weight through
\[
[w]_{A_\infty(\D)}:=\sup_{Q\in\D}\frac{1}{w(Q)}\int_{Q}\!M^{\D(Q)}w\,\mathrm{d}x.
\]
Finiteness of this constant characterizes the class $\bigcup_{p\geq 1}A_p(\D)$. Note that in particular we have
\[
[w]_{A_\infty(\D)}\leq[w]_{A_1(\D)}.
\]

In the non-dyadic case, the Fujii-Wilson characteristic is defined as
\begin{align*}
  [w]_{A_\infty}:=\sup_{Q}\frac{1}{w(Q)}\int_{Q}\!M(\ind_Q w)\,\mathrm{d}x.
\end{align*}
where the supremum is taken over all cubes.

We have the following well-known result (which, in fact, is a characterization of $A_\infty(\D)$):
\begin{proposition}\label{prop:wainfsparse}
Let $\Sp\subseteq\D$ be an $\eta$-sparse collection of cubes and let $Q_0\in\D$. Then for a weight $w$ we have
\[
  \sum_{\substack{Q\in\Sp\\ Q\subseteq Q_0}}w(Q)\leq \eta^{-1}[w]_{A_\infty(\D)} w(Q_0).
\]
\end{proposition}
\begin{proof}
We have
\[
\sum_{\substack{Q\in\Sp\\ Q\subseteq Q_0}}w(Q)\leq \eta^{-1}\sum_{\substack{Q\in\Sp\\ Q\subseteq Q_0}}\inf_{y\in Q}M^{\D(Q_0)}w(y)|E_Q|\leq \eta^{-1}\int_{Q_0}\!M^{\D(Q_0)}w\,\mathrm{d}x
\leq \eta^{-1}[w]_{A_\infty(\D)} w(Q_0),
\]
as desired.
\end{proof}

\subsection{Weighted BMO spaces}\label{sec:wBMO}
The space $\BMO(\R^d)$ is defined as the space of function $f\in L^0(\R^d)$ for which the sharp maximal function
\[
M^\sharp f:=\sup_Q\langle f-\langle f\rangle_Q\rangle_{1,Q}\ind_Q
\]
lies in $L^\infty(\R^d)$. Thus, it is sensible to define a weighted analogue $\BMO_w(\R^d)$ as those $f\in L^0(\R^d)$ for which $M^\sharp f\in L^\infty_w(\R^d)$. This is facilitated through the following definition.
\begin{definition}
For a weight $w$ we define the spaces $\BMO^1_w(\R^d)$ and $\BMO^\infty_w(\R^d)$ through
\begin{align*}
\|f\|_{\BMO^1_w(\R^d)}&:=\sup_{Q}\frac{\langle f-\langle f\rangle_Q\rangle_{1,Q}}{\langle w^{-1}\rangle_{1,Q}}=\sup_Q\frac{1}{w^{-1}(Q)}\int_Q\!|f-\langle f\rangle_Q|\,\mathrm{d}x,\\
\|f\|_{\BMO^\infty_w(\R^d)}&:=\sup_{Q}\langle w\rangle_{\infty,Q}\langle f-\langle f\rangle_Q\rangle_{1,Q},
\end{align*}
where functions are identified modulo constants.

We also define the weak analogues $\BMO^{1,\wk}_w(\R^d)$, $\BMO^{\infty,\wk}_w(\R^d)$ of these spaces through
\begin{align*}
\|f\|_{\BMO^{1,\wk}_w(\R^d)}&:=\sup_{Q}\inf_{c\in\C}\frac{\|(f-c)\ind_Q\|_{L^{1,\infty}(Q)}}{w^{-1}(Q)},\\
\|f\|_{\BMO^{\infty,\wk}_w(\R^d)}&:=\sup_{Q}\inf_{c\in\C}\langle w\rangle_{\infty,Q}|Q|^{-1}\|(f-c)\ind_Q\|_{L^{1,\infty}(Q)},
\end{align*}
where functions are identified modulo constants.

For a dyadic grid $\D$ in $\R^d$, the spaces $\BMO^{1}_w(\D)$, $\BMO^{\infty}_w(\D)$, $\BMO^{1,\wk}_w(\D)$, and $\BMO^{\infty,\wk}_w(\D)$ are defined analogously, this time taking the supremum over all cubes in $\D$, and functions are identified modulo constants on cubes in $\D$.
\end{definition}
Noting that
\[
\frac{\langle w\rangle_{\infty,Q}}{|Q|}=\frac{1}{(\essinf_Q w^{-1})|Q|}\geq\frac{1}{w^{-1}(Q)}
\]
for all cubes $Q$ in $\R^d$, we have
\[
\|f\|_{\BMO^1_w(\R^d)}\leq\|f\|_{\BMO^\infty_w(\R^d)},\quad \|f\|_{\BMO^{1,\wk}_w(\R^d)}\leq\|f\|_{\BMO^{\infty,\wk}_w(\R^d)}
\]
and hence, $\BMO^\infty_w(\R^d)\subseteq\BMO^1_w(\R^d)$ and  $\BMO^{\infty,\wk}_w(\R^d)\subseteq\BMO^{1,\wk}_w(\R^d)$. Conversely, if $w^{-1}\in A_1$, we have
\[
\langle w\rangle_{\infty,Q}\leq\frac{[w^{-1}]_{A_1}}{\langle w^{-1}\rangle_{1,Q}}
\]
for all cubes $Q$ in $\R^d$. Thus, in this case we have
\begin{align}
  \|f\|_{\BMO^\infty_w(\R^d)}&\leq[w^{-1}]_{A_1}\|f\|_{\BMO^1_w(\R^d)}, \label{eq:bmoinclusiona1} \\
  \|f\|_{\BMO^{\infty,\wk}_w(\R^d)}&\leq[w^{-1}]_{A_1}\|f\|_{\BMO^{1,\wk}_w(\R^d)},\label{eq:bmoinclusiona12}
\end{align}
and $\BMO^\infty_w(\R^d)=\BMO^1_w(\R^d)$, $\BMO^{\infty,\wk}_w(\R^d)=\BMO^{1,\wk}_w(\R^d)$.

We also note that $L^\infty_w(\R^d)\subseteq \BMO^1_w(\R^d)$ with
\[
\|f\|_{\BMO^1_w(\R^d)}\leq 2\|f\|_{L^\infty_w(\R^d)}.
\]
Thus, if $w^{-1}\in A_1$, we have $L^\infty_w(\R^d)\subseteq\BMO^\infty_w(\R^d)$ with
\[
\|f\|_{\BMO^\infty_w(\R^d)}\leq 2[w^{-1}]_{A_1}\|f\|_{L^\infty_w(\R^d)}.
\]

The following result shows that the space $\BMO_w^\infty(\R^d)$ consists of precisely those $f\in L^0(\R^d)$ for which $M^\sharp f\in L^\infty_w(\R^d)$.
\begin{proposition}\label{prop:msharpbmow}
We have
\[
\|f\|_{\BMO^\infty_w(\R^d)}= \|M^{\sharp}f\|_{L^\infty_w(\R^d)},\quad \|f\|_{\BMO^{\infty,\wk}_w(\R^d)}= \|M_{\wk}^{\sharp}f\|_{L_w^\infty(\R^d)},
\]
where
\[
M_{\wk}^{\sharp}f:=\sup_{Q}\inf_{c\in\C}|Q|^{-1}\|(f-c)\ind_Q\|_{L^{1,\infty}(Q)}\ind_Q.
\]
For a dyadic grid $\D$ in $\R^d$, analogously defining the sharp maximal operators $M^{\sharp,\D}$ and $M^{\sharp,\D}_{1,\infty}$ with respect to this grid, we also have
\[
\|f\|_{\BMO^\infty_w(\D)}= \|M^{\sharp,\D}f\|_{L^\infty_w(\R^d)},\quad \|f\|_{\BMO^{1,\infty}_w(\D)}= \|M_{\wk}^{\sharp,\D}f\|_{L_w^\infty(\D)}.
\]
\end{proposition}
\begin{proof}
We only prove the first equality, the others being analogous. Since $\ind_Qw\leq\langle w\rangle_{\infty,Q}$ for all cubes $Q$, we have
\[
(M^{\sharp}f)w=\sup_{Q}\langle f-\langle f\rangle_Q\rangle_{1,Q}\ind_Qw\leq\|f\|_{\BMO^\infty_w(\R^d)}.
\]
This proves the inequality $\|(M^{\sharp}f)w\|_{L^\infty(\R^d)}\leq \|f\|_{\BMO^\infty_w(\R^d)}$. For the converse, fix a cube $Q$ and let $\eps>0$. Then the set of $x\in Q$ such that $\langle w\rangle_{\infty,Q}\leq (1+\eps)w(x)$ has positive measure. Hence, there are $x\in Q$ for which
\[
\langle w\rangle_{\infty,Q}\langle f-\langle f\rangle_Q\rangle_{1,Q}\leq(1+\eps)\langle f-\langle f\rangle_Q\rangle_{1,Q}w(x)\ind_Q(x)\leq(1+\eps)\|(M^{\sharp}f)w\|_{L^\infty(\R^d)}
\]
Thus, $\|f\|_{\BMO^\infty_w(\R^d)}\leq(1+\eps)\|(M^{\sharp}f)w\|_{L^\infty(\R^d)}$. Letting $\eps\to 0$, the assertion follows.
\end{proof}

\begin{proposition}\label{prop:bmoconst}
Let $w$ be a weight. We have
\begin{align*}
\|f\|_{\BMO^1_w(\R^d)}&\eqsim \sup_{Q}
\inf_{c\in\C}\frac{\langle f-c\rangle_{1,Q}}{\langle w^{-1}\rangle_{1,Q}},\\
\|f\|_{\BMO^\infty_w(\R^d)}&\eqsim \sup_{Q}
\inf_{c\in\C}\langle w\rangle_{\infty,Q}\langle f-c\rangle_{1,Q}.
\end{align*}
Thus, in particular,
\begin{align*}
\|f\|_{\BMO_w^{1,\wk}(\R^d)}&\lesssim\|f\|_{\BMO^1_w(\R^d)},\\
\|f\|_{\BMO_w^{\infty,\wk}(\R^d)}&\lesssim\|f\|_{\BMO^\infty_w(\R^d)}.
\end{align*}
This similarly holds for $\BMO^1_w(\D)$, $\BMO^\infty_w(\D)$ for a dyadic grid $\D$ in $\R^d$.
\end{proposition}
\begin{proof}
We only prove the second equivalence, the first one being analogous. The inequality $\geq$ follows from taking $c=\langle f\rangle_Q$.  For the converse inequality,  fix a cube $Q$ and note that for any $c\in\C$ we have
\[
\langle f-\langle f\rangle_Q\rangle_{1,Q}\leq\langle f-c\rangle_{1,Q}+|\langle f-c\rangle_Q|\leq 2\langle f-c\rangle_{1,Q}.
\]
Thus, $\|f\|_{\BMO_w(\R^d)}\leq 2\sup_{Q}\inf_{c\in\C}\langle w\rangle_{\infty,Q}\langle f-c\rangle_{1,Q}$. As the result follows analogously for $\BMO_w(\D$), this proves the assertion.
\end{proof}

In the unweighted case, it follows from the John-Str\"omberg characterization of $\BMO$ that the weak $\BMO$ space is equivalent to the usual one. More precisely, denoting the non-decreasing rearrangement of a function $f$ by $f^\ast$, for a cube $Q$ and $\lambda\in(0,1)$ we define
\[
  M_\lambda^\sharp f:=\sup_Q\omega_\lambda(f;Q)\ind_Q,
\]
where
\begin{align*}
  \omega_\lambda(f;Q) &=\inf_{c\in\C}\inf\{t>0:|\{x\in Q:|f-c|>t\}|\leq\lambda|Q|\}.
\end{align*}

Observe that
\begin{align*}
  \omega_\lambda(f;Q) \leq \lambda^{-1}|Q|^{-1}\inf_{c\in\C}\|f-c\|_{L^{1,\infty}(Q)},
\end{align*}

so that $M^\sharp_\lambda f\leq \lambda^{-1}M^\sharp_{\wk} f$.

We can extend John-Str\"omberg's characterization of BMO through medians (see \cite{Stromberg}) to the weighted setting
using the following pointwise version by A. Lerner \cite[Theorem~1.3]{MR2010347}:
\begin{equation}\label{eq:jsequivalence}
M^\sharp f\eqsim_d M(M^\sharp_{\frac{1}{2}} f).
\end{equation}
for all $f\in L^1_{\loc}(\R^d)$.

\begin{proposition}
Let $w^{-1}\in A_1$. Then $\BMO^\infty_w(\R^d)=\BMO^{\infty,\wk}_w(\R^d)$ with
\[
\|f\|_{\BMO^\infty_w(\R^d)}\lesssim_d[w^{-1}]_{A_1}\|f\|_{\BMO^{\infty,\wk}_w(\R^d)}.
\]
\end{proposition}
\begin{proof}
By Proposition~\ref{prop:msharpbmow}, \eqref{eq:jsequivalence}, and Proposition~\ref{prop:winva1char}, we have
\[
\|f\|_{\BMO^\infty_w(\R^d)}\eqsim_d\|M(M^\sharp_{\frac{1}{2}}f)\|_{L^\infty_w(\R^d)}\leq[w^{-1}]_{A_1}\|M^\sharp_{\frac{1}{2}} f\|_{L^\infty_w(\R^d)}.
\]
The result now follows from the fact that $M^\sharp_{\frac{1}{2}}f\lesssim M^\sharp_{\wk}f$.
\end{proof}

\subsection{Banach function spaces}

We denote the positive measurable functions on $\R^d$ by $L^0(\R^d)_+$.
\begin{definition}
Suppose $\rho:L^0(\R^d)_+\to[0,\infty]$ satisfies
\begin{enumerate}[(i)]
\item\label{it:bfs1} $\rho(f)=0$ if and only if $f=0$ a.e.;
\item $\rho(f+g)\leq\rho(f)+\rho(g)$ and $\rho(\lambda f)=\lambda\rho(f)$ for all $f,g\in L^0(\R^d)_+$ and scalars $\lambda\geq 0$;
\item\label{it:bfs3} for every sequence $(f_n)_{n\in\N}$ in $L^0(\R^d)_+$ satisfying $0\leq f_n\uparrow f$ pointwise a.e. for $f\in L^0(\R^d)_+$, we have $\rho(f_n)\uparrow\rho(f)$.
\item\label{it:bfs4} for every measurable $E\subseteq\R^d$ of positive measure there exists a measurable $F\subseteq E$ of positive measure with $\rho(\ind_F)<\infty$.
\end{enumerate}
Then we call $\rho$ a function norm. Moreover, we call the space $X\subseteq L^0(\R^n)$ of $f\in L^0(\R^d)$ for which
\[
\|f\|_X:=\rho(|f|)<\infty,
\]
a Banach function space over $\R^d$.

Given a function norm $\rho$, we define a new function norm $\rho'$ through
\[
\rho'(g):=\sup_{\rho(f)= 1}\|fg\|_{L^1(\R^d)}.
\]
The K\"ote dual of $X$ is defined as the Banach function space associated to $\rho'$, i.e.,
\[
X':=\{g\in L^0(\R^d):fg\in L^1(\R^d)\text{ for all $f\in X$}\}
\]
with
\[
\|g\|_{X'}=\sup_{\|f\|_X=1}\|fg\|_{L^1(\R^d)}.
\]
\end{definition}

Property \ref{it:bfs4} is called the saturation property of $\rho$ and is equivalent to the existence of a weak order unit, i.e., an $f>0$ for which $\rho(f)<\infty$. Property \ref{it:bfs4} is also equivalent to the fact that $\rho'$ satisfies \ref{it:bfs1} and hence, is a function norm. Thus, $X'$ is a Banach function space over $\R^d$.

The Fatou property \ref{it:bfs3} is equivalent to the assertion that $\rho''=\rho$ and hence, $X''=X$ isometrically. Thus, in particular we have
\[
\|f\|_X=\sup_{\|g\|_{X'}=1}\|fg\|_{L^1(\R^d)}.
\]

We say that a Banach function space $X$ over $\R^d$ is order-continuous when for all $(f_n)_{n\in\N}$ in $X$ with $f_n\downarrow 0$ a.e. we have $\|f_n\|_X\to 0$. We note that $X$  is reflexive if and only if $X$ and $X'$ are order-continuous.
For proofs of the above claims we refer the reader to \cite{ZaanenBook}.

By carefully tracking the constants in the proof of \cite[Corollary~4.3]{MR2763004}, we arrive at the following quantitative version:
\begin{theorem}\label{thm:fefsteinbfs}
Let $X$ be a Banach function space over $\R^d$ and suppose that $M:X\to X$ is bounded. Then the following assertions are equivalent:
\begin{enumerate}[(i)]
\item\label{it:fscor1} There is a constant $C_X>0$ such that
\[
\|f\|_X\leq C_X\|M^\sharp f\|_X
\]
for all $f\in X$ with the property that $|\{x\in\R^d:|f(x)|>\alpha\}|<\infty$ for all $\alpha>0$;
\item\label{it:fscor2} There is a constant $C_{X,\wk}>0$ such that
\[
\|f\|_X\leq C_{X,\wk}\|M_{\wk}^\sharp f\|_X
\]
for all $f\in X$ with the property that $|\{x\in\R^d:|f(x)|>\alpha\}|<\infty$ for all $\alpha>0$;
\item\label{it:fscor3} $M:X'\to X'$ is bounded.
\end{enumerate}
Moreover, if $C_X$ and $C_{X,\wk}$ are the smallest possible constant in \ref{it:fscor1} and \ref{it:fscor2}, then
\begin{equation}\label{eq:fscor1}
C_X\leq C_{X,\wk}\lesssim_d\|M\|_{X'\to X'}\lesssim_d C_X\|M\|_{X\to X}.
\end{equation}
\end{theorem}
\begin{proof}
Our strategy will be to prove that \ref{it:fscor2}$\Rightarrow$\ref{it:fscor1}$\Rightarrow$\ref{it:fscor3}$\Rightarrow$\ref{it:fscor2}.

The assertion \ref{it:fscor2}$\Rightarrow$\ref{it:fscor1} with $C_X\leq C_{X,\wk}$ follows from the fact that $M^\sharp_{\wk}f\leq M^\sharp f$.

For \ref{it:fscor1}$\Rightarrow$\ref{it:fscor3} we use that it was shown in \cite[Theorem 1.1]{MR2763004} that \ref{it:fscor1} is equivalent to the assertion
\begin{enumerate}[(i)]
\setcounter{enumi}{3}
\item\label{it:fscor4} there is a $C>0$ such that
\[
\|fMg\|_{L^1(\R^d)}\leq C\|Mf\|_X\|g\|_{X'}
\]
for all $f\in L^1_{\loc}(\R^d)$ and $g\in X'$.
\end{enumerate}
Moreover, if $C$ is the smallest possible constant in this estimate, then
\[
C\eqsim_d C_X.
\]
Thus, we have
\[
\|Mg\|_{X'}=\sup_{\|f\|_X=1}\|fMg\|_{L^1(\R^d)}\leq C\sup_{\|f\|_X=1}\|Mf\|_X\|g\|_{X'}=C\|M\|_{X\to X}\|g\|_{X'}
\]
so that $M:X'\to X'$ with $\|M\|_{X'\to X'}\lesssim_d C_X\|M\|_{X\to X}$, as desired

For \ref{it:fscor3}$\Rightarrow$\ref{it:fscor2}, we use \cite[Theorem 1]{MR2093912} which states that there is a dimensional constant $\lambda\in(0,1)$ such that for all $f$ satisfying the property that $|\{x\in\R^d:|f(x)|>\alpha\}|<\infty$ for all $\alpha>0$ and all $g\in L^1_{\loc}(\R^d)$ we have
\[
\|fg\|_{L^1(\R^d)}\lesssim_d\|(M^\sharp_\lambda f)(Mg)\|_{L^1(\R^d)}.
\]
Since $M^\sharp_\lambda f\leq\lambda^{-1} M^\sharp_{\wk}f$, it follows that
\begin{align*}
\|f\|_X&=\sup_{\|g\|_{X'}=1}\|fg\|_{L^1(\R^d)}
\lesssim_d\sup_{\|g\|_{X'}=1}\|(M^\sharp_{\wk}f)(Mg)\|_{L^1(\R^d)}\\
&\leq\sup_{\|g\|_{X'}=1}\|M^\sharp_{\wk}f\|_X\|Mg\|_{X'}=\|M\|_{X'\to X'}\|M^\sharp_{\wk}f\|_X.
\end{align*}
This proves \ref{it:fscor2} with $C_{X,\wk}\lesssim_d\|M\|_{X'\to X'}$ and \eqref{eq:fscor1}. The assertion follows.
\end{proof}

\section{Rubio de Francia extrapolation from weighted BMO}\label{section:extrapolation}

\begin{theorem}\label{thm:sharpbmoextrap}
  Let $T$ be an operator in $L^0(\R^d)$.
  Suppose there is an increasing function $\phi:[1,\infty)\to(0,\infty)$ such that for all weights $w^{-1}\in A_1$ and all compactly supported functions $f\in L_w^\infty(\R^d)$ we have
  \begin{align} \label{thm:sharpbmoextrap:1}
    \|Tf\|_{\BMO_w^{\infty}(\R^d)}\leq\phi([w^{-1}]_{A_1})\|f\|_{L^\infty_w(\R^d)}.
  \end{align}
  Then, for all Banach function spaces $X$ over $\R^d$ for which $M$ is bounded on $X$ and $X'$, and all compactly supported functions $f\in X$ we have
  \[
    \|Tf\|_X\lesssim_d C_X\phi(2\|M\|_{X\to X})\|f\|_X,
  \]
  where $C_X$ is the constant from Theorem \ref{thm:fefsteinbfs}.
  If $T$ is a (sub)linear operator and $X$ is order-continuous, then $T$ extends to a bounded operator $X\to X$ satisfying
  \begin{align} \label{thm:sharpbmoextrap:2}
    \|T\|_{X\to X}\lesssim_d C_X \phi(2\|M\|_{X\to X}),
  \end{align}
  in particular
  \begin{align} \label{thm:sharpbmoextrap:3}
    \|T\|_{X\to X}\lesssim_d\|M\|_{X'\to X'}\phi(2\|M\|_{X\to X}).
  \end{align}

  An analogous assertion holds when you replace $\BMO_w^\infty(\R^d)$ by $\BMO_w^{\infty,\wk}(\R^d)$ and $C_X$ by $C_{X, \wk}$.
\end{theorem}

We need the following lemma which is based on the Rubio de Francia algorithm.
\begin{lemma}\label{lem:rdf}
Let $X$ be a Banach function space over $\R^d$ for which $M:X\to X$. Then for all $f\in X$, $g\in X'$ there exists a weight $w^{-1}\in A_1$ such that
\begin{equation}\label{eq:rdf1}
[w^{-1}]_{A_1}\leq 2\|M\|_{X\to X},
\end{equation}
and $f\in L^\infty_w(\R^d)$, $g\in L^1_{w^{-1}}(\R^d)$ with
\begin{equation}\label{eq:rdf2}
\|f\|_{L^\infty_w(\R^d)}\|g\|_{L_{w^{-1}}^1(\R^d)}\leq 2\|f\|_X\|g\|_{X'}.
\end{equation}
\end{lemma}
\begin{proof}
We set
\[
w^{-1}:=\sum_{k=0}^\infty\frac{M^kf}{2^k\|M\|^k_{X\to X}},
\]
where we have recursively defined $M^0f:=|f|$, $M^{k+1}f:=M(M^kf)$. Then we have
\[
M(w^{-1})\leq\sum_{k=0}^\infty \frac{M^{k+1}f}{2^k\|M\|_{X\to X}^k}\leq 2\|M\|_{X\to X}w^{-1},
\]
proving \eqref{eq:rdf1}.

Since $|f|=M^0f\leq w^{-1}$, we have $\|f\|_{L^\infty_w(\R^d)}\leq 1$. Combining this with the fact that $\|w^{-1}\|_X\leq 2\|f\|_X$, we have
\[
\|f\|_{L^\infty_w(\R^d)}\|g\|_{L_{w^{-1}}^1(\R^d)}\leq \|gw^{-1}\|_{L^1(\R^d)}\leq\|w^{-1}\|_X\|g\|_{X'}\leq2\|f\|_X\|g\|_{X'}.
\]
This proves \eqref{eq:rdf2}, as desired.
\end{proof}
\begin{proof}[Proof of Theorem~\ref{thm:sharpbmoextrap}]
  Let $f\in X$ have compact support and let $g\in X'$ with $\|g\|_{X'}=1$. By Lemma~\ref{lem:rdf} we can pick a weight $w^{-1}\in A_1$ such that $[w^{-1}]_{A_1}\leq 2\|M\|_{X\to X}$ and $f\in L^\infty_w(\R^d)$, $g\in L^1_{w^{-1}}(\R^d)$ with
  \[
    \|f\|_{L^\infty_w(\R^d)}\|g\|_{L^1_{w^{-1}}(\R^d)}\leq 2\|f\|_X.
  \]
  Hence, by Proposition~\ref{prop:msharpbmow} we have
  \begin{align*}
    \|M^\sharp(Tf)g\|_{L^1(\R^d)}&\leq\|M^\sharp(Tf)\|_{L^\infty_w(\R^d)}\|g\|_{L^1_{w^{-1}}(\R^d)}\leq\phi([w^{-1}]_{A_1})\|f\|_{L^\infty_w(\R^d)}\|g\|_{L^1_{w^{-1}}(\R^d)}\\
    &\leq 2\phi(2\|M\|_{X\to X})\|f\|_X.
  \end{align*}
  Hence, since $f$ has compact support and thus $|\{|f|>\lambda\}|\leq|\supp f|<\infty$ for all $\lambda>0$, it follows from Theorem~\ref{thm:fefsteinbfs}, that
  \begin{align*}
    \|Tf\|_X &\lesssim_d C_X M^\sharp(Tf)\|_X = C_X \sup_{\|g\|_{X'}=1}\|M^\sharp(Tf)g\|_{L^1(\R^d)}\\
    &\leq2 C_X \phi(2\|M\|_{X\to X})\|f\|_X,
  \end{align*}
  proving the first result.

  For the next assertion, since $X$ is order continuous the functions in $X$ of compact support are dense in $X$.
  Indeed, given $f\in X$, set $f_n:=f\ind_{B(0;n)}$. Then $f_n\in X$ has compact support and $|f-f_n|=|f|\ind_{\R^d\backslash B(0;n)}\downarrow 0$.
  Hence, $\|f-f_n\|_X\to 0$, as desired.
  Thus, the result follows from the fact that if $T$ is (sub)linear, then it is is uniformly continuous and hence, extends to all of $X$ with the same bound.

  The assertion for 
  $\BMO_w^\infty(\R^d)$ replaced by $\BMO_w^{\infty,\wk}(\R^d)$ and $C_X$ replaced by $C_{X, \wk}$ is proved analogously.
\end{proof}

\begin{remark}\label{rem:sharpextrap}
  It was shown in \cite{CPR} that the initial estimate \eqref{thm:sharpbmoextrap:1} implies that for $p \in (1,\infty)$ and $w\in A_p$ we have
  \[
    \|T\|_{L^p(w)}\lesssim_d\|M\|_{L^{p'}(\R^d;w^{1-p'})}\phi(\|M\|_{L^p(\R^d;w)}).
  \]
  We recover this result in \eqref{thm:sharpbmoextrap:3},
  since if $X=L^p(\R^d;w)$, then $X'=L^{p'}(\R^d;w^{1-p'})$.
  As a matter of fact, we improve on this estimate in \eqref{thm:sharpbmoextrap:2}.
  Indeed, for $X=L^p(w;\R^d)$ we obtain 
  \[
    \|T\|_{L^p(w)}\lesssim_d C_X\phi(\|M\|_{L^p(\R^d;w)}),
  \]
  where $C_X$ is the constant appearing in the Fefferman-Stein inequality
  \[
    \|f\|_{L^p(\R^d;w)}\leq C_X\|M^\sharp f\|_{L^p(\R^d;w)}.
  \]
  While $\|M\|_{L^{p'}(\R^d;w^{1-p'})}$ is bounded if and only if $w\in A_p$,
  the constant $C_X$ is bounded under much more general conditions.
  For example, it is bounded when $w\in A_\infty$.
\end{remark}

\section{Weighted BMO estimates for singular integrals} \label{section:czo}

We will begin this section discussing the action of sparse operators on $L^\infty_w$. These serve as a
way to gain intuition for the more complicated Calder\'on-Zygmund operators. We note, however, that unlike
in the classical $L^p$ case, estimates for sparse operators do not immediately imply bounds for Calder\'on-Zygmund operators
because BMO is not a lattice.

\subsection{Sparse operators}
For $\eta\in(0,1)$, a collection of sets $\Sp$ in $\R^d$ is called $\eta$-\emph{sparse} if there exists a pairwise disjoint collection of sets $\{E_S\}_{S\in\Sp}$ such that $E_S\subseteq S$ and $|E_S|\geq\eta|S|$ for all $S\in\Sp$.
For such a collection we define the \emph{sparse operator} associated to $\Sp$ as
\[
\A_{\Sp}f:=\sum_{Q\in\Sp}\langle f\rangle_{1,Q}\ind_Q.
\]

In general, these operators \emph{do not} map $L^\infty$ to $BMO$. There are several obstacles that prevent this,
which we now describe.

For any family $\Sp$ of sets in $\R^d$, let $h_{\Sp}$ be its height function:
\begin{align*}
  h_{\Sp} = \sum_{J \in \Sp} \ind_J.
\end{align*}

First we note that $\A_{\Sp}f$ may not even be locally integrable for $f \in L^\infty(\R^d)$. In fact, the situation is much worse.
Indeed, consider
the sets $F_n = [0,1) \cup [n, n+1)$ for $n \geq 1$ and define the family $\Sp = \{F_n: n \geq 1\}$. For each $n \geq 1$ the subset $E_n = [n,n+1) \subseteq F_n$
has measure $1 = \frac{1}{2}|F_n|$ and the collection $\{E_n\}$ is pairwise disjoint, thus $\Sp$ is $\frac{1}{2}$-sparse.
However
\begin{align*}
  \A_{\Sp}(\ind_{[0,1)}) = \sum_{n=1}^\infty \frac{1}{2}\ind_{F_n} = \frac{1}{2}h_{\Sp}
\end{align*}
is identically $\infty$ on $[0,1)$.

In general, we need \emph{some} boundedness of the maximal function associated to $\Sp$ for $\A_{\Sp}$ to behave well. In particular,
if $\Sp$ consists of intervals (or generally cubes in any dimension), then $\A_{\Sp}$ is bounded on $L^2(\R^d)$, so we have no problems
with the local integrability of $\A_{\Sp}(f)$. But even this is not enough to reasonably define $\A_{\Sp}$ as an operator from $L^\infty(\R^d)$ to $\BMO(\R^d)$.

To see this, consider the family
$\Sp = \{[0,2^n): n \geq 0\}$.
This family is $\frac{1}{2}$-sparse, however if we define $\A_{\Sp}(1)$ with the natural pointwise limit we obtain
$\A_{\Sp}(1) = \infty$ on $[0, \infty)$.

In general, we need $\Sp$ to have finite total measure, at least qualitatively. However, this is also not enough because there may still be issues similar to those that appear
when comparing $\BMO$ and dyadic $\BMO$. Indeed, consider
\begin{align*}
  \Sp = \{[0,2^{-n}): n \geq 0\}.
\end{align*}
This family is $\frac{1}{2}$-sparse and has finite total measure, thus $\A_{\Sp}(f)$ is integrable for all $f \in L^\infty(\R^d)$. However,
letting $I_n = [0, 2^{-n})$ and $J_n = [-2^{-n}, 2^{-n})$
for $n \geq 1$, we have
\begin{align*}
  \frac{1}{|J_n|} \int_{J_n}\!\A_{\Sp}(\mathbbm{1}_{[0,1)})\,\mathrm{d}x &= 2^{n-1} \sum_{m=0}^\infty |I_m  \cap I_n| \\
  &= 2^{n-1} \sum_{m=0}^{n} |I_n| + 2^{n-1} \sum_{m=n+1}^\infty |I_m| \\
  &= 2^{n-1}(n+1)2^{-n} + 2^{n-1} \sum_{m=n+1}^\infty 2^{-m} \\
  &= \frac{n}{2} + 1.
\end{align*}
Thus
\begin{align*}
  \frac{1}{|J_n|}\int_{J_n}\! |\A_{\Sp}(\ind_{[0,1)}) - \langle \A_{\Sp}(\mathbbm{1}_{[0,1)}) \rangle_{J_n}|\,\mathrm{d}x
    &\geq \frac{1}{|J_n|}\int_{-2^{-n}}^0\! \bigl(\frac{n}{2}+1 \bigr)\,\mathrm{d}x \\
    &\geq \frac{n}{4}
\end{align*}
and hence, $\A_{\Sp}(\ind_{[0,1)})$ is not in $\BMO(\R^d)$.

When the sparse collection consists of cubes belonging to the same dyadic family $\D$, then we \emph{can} obtain $\BMO(\D)$ estimates.
\begin{theorem}\label{thm:wbmosparse}
  Suppose $w^{-1}\in A_1$ and let $\Sp\subseteq\D$ be an $\eta$-sparse collection of cubes. Then
  \begin{align}\label{eq:wbmosparse1}
    \|A_{\Sp}f\|_{\BMO^1_w(\D)}&\lesssim \eta^{-1}[w^{-1}]_{A_\infty(\D)}\|f\|_{L^\infty_w(\R^d)},\\\label{eq:wbmosparse2}
    \|A_{\Sp}f\|_{\BMO^\infty_w(\D)}&\lesssim \eta^{-1}[w^{-1}]_{A_1(\D)}[w^{-1}]_{A_\infty(\D)}\|f\|_{L^\infty_w(\R^d)}
  \end{align}
  and
  \begin{align}\label{eq:wbmosparse3}
    \|A_{\Sp}f\|_{\BMO^{1,\wk}_w(\D)}&\lesssim \eta^{-1} \|f\|_{L^\infty_w(\R^d)},\\\label{eq:wbmosparse4}
    \|A_{\Sp}f\|_{\BMO^{\infty,\wk}_w(\D)}&\lesssim \eta^{-1} [w^{-1}]_{A_1(\D)}\|f\|_{L^\infty_w(\R^d)}.
  \end{align}
\end{theorem}
\begin{proof}
  Fix $Q_0\in\D$ and let $c:=\sum_{\substack{Q\in\Sp\\Q\supseteq Q_0}}\langle f\rangle_{1,Q}$. Then, by Proposition~\ref{prop:wainfsparse} we have
  \begin{align*}
    \langle A_{\Sp}f-c\rangle_{1,Q_0}&=\frac{1}{|Q_0|}\sum_{\substack{Q\in\Sp\\Q\subseteq Q_0}}\int_Q\!f\,\mathrm{d}x
    \leq \eta^{-1}\|f\|_{L^\infty_w(\R^d)}\frac{1}{|Q_0|}\sum_{\substack{Q\in\Sp\\Q\subseteq Q_0}}w^{-1}(Q)\\
    &\leq \eta^{-1}[w^{-1}]_{A_\infty(\D)}\|f\|_{L^\infty_w(\R^d)}\langle w^{-1}\rangle_{1,Q_0}.
  \end{align*}
  Thus, \eqref{eq:wbmosparse1} follows from Proposition~\ref{prop:bmoconst}. The estimate \eqref{eq:wbmosparse2} then follows from \eqref{eq:bmoinclusiona1}.

  For \eqref{eq:wbmosparse3} we note that since $\|A_{\Sp(Q_0)}\|_{L^1(Q_0)\to L^{1,\infty}(Q_0)}\lesssim\eta^{-1}$, we have
  \[
    \|A_{\Sp}f-c\|_{L^{1,\infty}(Q_0)}=\|A_{\Sp(Q_0)} f\|_{L^{1,\infty}(Q_0)}
    \lesssim \eta^{-1} \|f\|_{L^1(Q_0)}
    \leq \eta^{-1}\|f\|_{L^\infty_w(\R^d)} w^{-1}(Q_0),
  \]
  as desired. The inequality \eqref{eq:wbmosparse4} then follows from \eqref{eq:bmoinclusiona12}.
\end{proof}

\subsection{Calder\'on-Zygmund operators}
We now turn our study to $\BMO$ bounds for Calder\'on-Zygmund operators. In particular, we consider operators $T$ which are weak-type $(1,1)$-bounded operator and are given by
\begin{align} \label{czo:rep}
  Tf(x) = \int_{\R^d} K(x,y) f(y) \dd y \quad \text{for all }x \notin \supp f
\end{align}
for all $f$ in $L^1(\R^d)$ with compact support.

We will furthermore suppose that $K$ satisfies the following smoothness condition:
\begin{align*}
  |K(x,y) - K(z,y)| \leq \Omega\biggl( \frac{|x-z|}{|x-y|} \biggr) \frac{1}{|x-y|^d}
\end{align*}
for all $x,y,z$ satisfying $|x-y| > 2|x-z| > 0$, and where $\Omega: [0,\infty) \to [0, \infty)$ is an increasing subadditive function with $\Omega(0) =0$.
We can quantify the smoothness of the kernel in terms of the Dini condition of $\Omega$:
\begin{align*}
  \|\Omega\|_{\Dini} := \int_0^1 \Omega(t) \, \frac{\mathrm{d} t}{t}.
\end{align*}

We will be interested in $L^\infty_w \to BMO_w$ estimates for $T$, so we should be careful about whether $Tf$ is well-defined.
Indeed, if $f \in L^\infty_w(\R^d)$ has compact support, and $w^{-1} \in L^1_{\loc}(\R^d)$, then $f \in L^1(\R^d)$,
so we can use the representation \eqref{czo:rep}.

To study the $L^\infty \to \BMO$ boundedness of such operators we will employ the following pointwise domination theorem. Originally,
a similar theorem was proved by A. Lerner in \cite{LernerPointwise}, but we will use the following version by T. Hyt\"onen:
\begin{theorem}[{\cite[Theorem 2.3]{Remarks}}] \label{lerners_formula}
  For any measurable function $f$ on a cube $Q_0$ in $\mathbb{R}^d$ there exists a sparse collection $\mathcal{S} \subseteq Q_0$ such that
  \begin{align*}
    |f(x) - m_{Q_0}(f)| \leq 2 \sum_{Q \in \mathcal{S}} \omega_{\lambda}(f;Q) \ind_Q(x),
  \end{align*}
  where $\lambda = 2^{-d-2}$ and $m_{Q_0}(f)$ is any of the possible medians of $f$ over $Q$, i.e.: constants $m$ such that
  \begin{align*}
    |\{x \in Q_0:\, f(x) > m \}| \leq \frac{1}{2}|Q| \quad\text{and}\quad |\{x \in Q_0:\, f(x) < m \}| \leq \frac{1}{2}|Q|.
  \end{align*}
\end{theorem}
The idea is to use this theorem, but applied to $Tf$. One can estimate $\omega_\lambda(Tf;Q)$ in terms of averages of $f$, in particular we have
the following estimate which goes back to \cite{JT}:
\begin{align} \label{czo:osc_estimate}
  \omega_{\lambda}(Tf;Q) \lesssim_d \sum_{m=0}^\infty \Omega(2^{-m})\langle |f| \rangle_{2^m Q}.
\end{align}
This, together with Theorem \ref{lerners_formula}, yields the following: for every compatly supported function $f \in L^1(\R^d)$ and every cube $Q_0$ there exists
a sparse collection $\mathcal{S}$ of cubes in $Q_0$ such that
\begin{align} \label{czo:pointwise_TF}
  |Tf(x) - m_{Q_0}(Tf)| \lesssim_d \sum_{m=0}^\infty \Omega(2^{-m}) \sum_{Q \in \mathcal{S}} \langle |f| \rangle_{2^mQ} \ind_Q(x).
\end{align}

We will be able to prove certain weighted estimates for such operators, but unless $w^{-1} \in A_1$ the estimates will depend
on the interaction between $\Omega$ and the weight. We will quantify this interaction with a variant
of the $B_2$ condition
\begin{align*}
  [w]_{B_2} = \sup_{I} \frac{|I|^2}{w(I)} \int_{I^c} \frac{w(x)}{|x-c_I|^2} \dd x,
\end{align*}
where the supremum is taken over all intervals $I$ and $c_I$ is the center of $I$.

This class was introduced originally in \cite{HMW_Hilbert}, and then in \cite{MW_Hilbert} where it was part of the characterization
of the $L^\infty_w \to BMO_w^1$ boundedness of the Hilbert transform.
In our situation we need to adapt this condition to incorporate the Dini modulus:
\begin{align*}
  [w]_{B(\Omega)} = \sup_Q \frac{|Q|}{w(Q)} \int_{Q^c}\frac{w(x)}{|x-c_Q|^d}\Omega\Bigl( \frac{\ell(Q)}{|x-c_Q|}\Bigr) \dd x
\end{align*}
where the supremum is taken over all cubes $Q$.
Note that in dimension one and when $\Omega(x) = x$ this reduces to the $B_2$ condition.

We can prove a result analogous to Lemma 1 from \cite{HMW_Hilbert} which gives sufficient conditions for a weight to belong
to $B(\Omega)$.
\begin{proposition} \label{czo:embedding}
  Let $w \in A_p$ and suppose $\Omega$ is as above, then
  \begin{align*}
    [w]_{B(\Omega)} \lesssim_d [w]_{A_p} \int_0^1 \Omega(t) t^{-d(p-1)} \, \frac{\mathrm{d} t}{t}.
  \end{align*}
  In particular, if $w \in A_1$ then
  \begin{align*}
    [w]_{B(\Omega)} \lesssim_d [w]_{A_1} \|\Omega\|_{\Dini}.
  \end{align*}
\end{proposition}
\begin{proof}
  Fix a cube $Q$, and set $S_m = \{x:\, 2^{m-1} \ell(Q) \leq |x-c_Q| < 2^m \ell(Q)\}$
  \begin{align} \label{czo:ap}
    \int_{Q^c} \frac{w(x)}{|x-c_Q|^d} \Omega\biggl( \frac{\ell(Q)}{|x-c_Q|} \biggr) \dd x &= \sum_{m=1}^\infty \int_{S_m} \frac{w(x)}{|x-c_Q|^d} \Omega\biggl( \frac{\ell(Q)}{|x-c_Q|} \biggr) \dd x \notag \\
    &\leq \sum_{m=1}^\infty \int_{S_m} \frac{w(x)}{(2^{m-1} \ell(Q))^d} \Omega(2^{-(m-1)}) \dd x \notag \\
    &= \sum_{m=0}^\infty \frac{ \Omega(2^{-m})}{(2^m \ell(Q))^d} \int_{S_{m+1}} w(x) \dd x  \notag\\
    &\leq \sum_{m=0}^\infty \frac{ \Omega(2^{-m})}{(2^m \ell(Q))^d} \int_{2^{m+2}Q} w(x) \dd x  \notag\\
    &= 2^{2d}  \sum_{m=0}^\infty  \Omega(2^{-m})\frac{1}{|2^{m+2} Q|} \int_{2^{m+2}Q} w(x) \dd x.
  \end{align}

  Note that by definition
  \begin{align*}
    \frac{1}{|2^{m+2}Q|} \int_{2^{m+2}Q }w \dd x &\leq [w]_{A_p} \langle w^{1-p'} \rangle_{2^{m+2}Q}^{1-p} \\
    &\leq [w]_{A_p} 2^{(m+2)d(p-1)} \langle w \rangle_Q
  \end{align*}
  so replacing in \eqref{czo:ap}:
  \begin{align*}
    \frac{|Q|}{w(Q)}\int_{Q^c} \frac{w(x)}{|x-c_Q|^d} \Omega\biggl( \frac{\ell(Q)}{|x-c_Q|} \biggr) \dd x &\lesssim_d [w]_{A_p}\sum_{m=0}^\infty \Omega(2^{-m}) 2^{md(p-1)}  \\
    &\eqsim_d [w]_{A_p} \int_0^1 \Omega(t) t^{-d(p-1)} \,\frac{\mathrm{d}t}{t}
  \end{align*}
\end{proof}
\begin{remark*}
  When $T$ is the Hilbert transform (or Riesz transforms) then $\Omega(t) = t$ and then the condition becomes
  \begin{align*}
    [w]_{B(\Omega)} \lesssim_d [w]_{A_p} \int_0^1 t^{-d(p-1)} \dd t,
  \end{align*}
  which is finite when $p < 1 + \frac{1}{d}$. In particular, if $w \in A_{1+\frac{1}{d}}$ then, by self-improvement we can deduce that $w$ is in $B(\Omega)$.
  This recovers Lemma 1 from \cite{HMW_Hilbert}.
\end{remark*}

\begin{theorem} \label{czo:theorem}
  Let $w^{-1} \in L^1_{\loc}(\R^d)$ and let $T$ be an operator as above, then
  \begin{align}
    \|Tf\|_{\BMO_w^1(\mathbb{R}^d)} &\lesssim \bigl([1]_{B(\Omega)} + [w^{-1}]_{B(\Omega)}\bigr)[w^{-1}]_{A_\infty}  \|f\|_{L_w^\infty(\R^d)} \label{czo:1} \\
    \|Tf\|_{\BMO_w^\infty(\mathbb{R}^d)} &\lesssim [w^{-1}]_{A_1}^2[w^{-1}]_{A_\infty}\|\Omega\|_{\Dini}  \|f\|_{L_w^\infty(\R^d)} \label{czo:2}
  \end{align}
  for all functions $f \in L^\infty_w(\R^d)$ with compact support.
\end{theorem}
\begin{proof}
  Take $f \in L^\infty$ with compact support and fix a cube $Q_0$. Combining \eqref{czo:osc_estimate} and \eqref{czo:pointwise_TF}
  we have
  \begin{align*}
    |Tf(x) - m_{Q_0}(Tf)| \lesssim \sum_{Q \in \mathcal{S}} \sum_{m=0}^\infty \Omega(2^{-m}) \langle |f| \rangle_{2^m Q} \mathbbm{1}_Q(x)
  \end{align*}
  for some sparse collection $\mathcal{S} \subset \mathcal{D}(Q_0)$.

  Assume without loss of generality that $fw \leq 1$, then
  \begin{align*} 
    \int_{Q_0}|Tf - m_{Q_0}(Tf)| \dd x&\lesssim \sum_{Q \in \mathcal{S}} \sum_{m=0}^\infty \Omega(2^{-m}) \langle w^{-1} \rangle_{2^mQ} |Q| \\
    &= \sum_{Q \in \mathcal{S}} \sum_{m=0}^\infty \Omega(2^{-m}) 2^{-md} \int_{2^m Q} w^{-1} \dd x.
  \end{align*}
  We first prove \eqref{czo:1}. Setting $\sigma = w^{-1}$:
  \begin{align} \label{czo:inter}
    \sum_{m=0}^\infty \Omega(2^{-m}) 2^{-md} \int_{2^m Q} \sigma \dd x = \int_{\mathbb{R}^d} \sigma \sum_{m=0}^\infty \Omega(2^{-m}) 2^{-md} \mathbbm{1}_{2^mQ}\dd x.
  \end{align}

  When $x \in Q$ then
  \begin{align*}
    \sum_{m=0}^\infty \Omega(2^{-m}) 2^{-md} \mathbbm{1}_{2^mQ}(x) = \mathbbm{1}_Q(x)\sum_{m=0}^\infty \Omega(2^{-m})2^{-md} \lesssim_d [1]_{B(\Omega)} \mathbbm{1}_Q(x).
  \end{align*}

  When $x \in 2^kQ \setminus 2^{k-1}Q$, for $k \geq 1$:
  \begin{align*}
    \sum_{m=0}^\infty \Omega(2^{-m}) 2^{-md} \mathbbm{1}_{2^mQ}(x) &= \sum_{m=k}^\infty \Omega(2^{-m}) 2^{-md} \\
    &\leq \Omega(2^{-k}) \sum_{m=k}^\infty 2^{-md} \\
    &\lesssim  \Omega(2^{-k}) 2^{-kd} \\
    &\lesssim \Omega\Bigl(\frac{\ell(Q)}{|x-c_Q|}\Bigr) \frac{|Q|}{|x-c_Q|^d}.
  \end{align*}

  Putting these estimates together, and going back to \eqref{czo:inter}:
  \begin{align*}
    \int_{\mathbb{R}^d} \sigma \sum_{m=0}^\infty \Omega(2^{-m}) 2^{-md} \mathbbm{1}_{2^mQ}\dd x
    &\lesssim_d  [1]_{B(\Omega)}\int_{Q} \sigma \dd x + |Q|\int_{Q^c} \frac{\sigma(x)}{|x-c_Q|^d} \Omega\Bigl( \frac{\ell(Q)}{|x-c_Q|} \Bigr) \dd x \\
    &\eqsim \bigl([1]_{B(\Omega)} + [\sigma]_{B(\Omega)}\bigr) \sigma(Q).
  \end{align*}

  Finally, we can use this estimate to finish with Proposition \ref{prop:wainfsparse}
  \begin{align*}
    \sum_{Q \in \mathcal{S}} \sum_{m=0}^\infty \Omega(2^{-m}) 2^{-md} \int_{2^m Q}\sigma &\lesssim \bigl([1]_{B(\Omega)} + [\sigma]_{B(\Omega)}\bigr) \sum_{Q \in \mathcal{S}} \sigma(Q) \\
    &\lesssim \bigl([1]_{B(\Omega)} + [\sigma]_{B(\Omega)}\bigr) [\sigma]_{A_\infty} \sigma(Q_0).
  \end{align*}
  Putting everything together, we have shown
  \begin{align*}
    \frac{1}{w^{-1}(Q_0)} \int_{Q_0} |Tf - m_{Q_0}(Tf)| \dd x \lesssim 
    \bigl([1]_{B(\Omega)} + [w^{-1}]_{B(\Omega)}\bigr)[w^{-1}]_{A_\infty} 
  \end{align*}
  assuming $|f|w \leq 1$, so \eqref{czo:1} follows.
  One can deduce \eqref{czo:2} by combining \eqref{czo:1} together with \eqref{eq:bmoinclusiona1} and Proposition \ref{czo:embedding}.

\end{proof}

\subsection*{Acknowledgements}
Z.N. is supported by the Basque Government through the BERC 2018-2021 program and by the Spanish State Research Agency through BCAM Severo
Ochoa excellence accreditation SEV-2017-0718 and through project PID2020-113156GB-I00/AEI/10.13039/501100011033 funded by Agencia Estatal de
Investigación and acronym ``HAPDE''.

G.R. was supported in part by Grant MICIN/AEI/PID2019-105599GB-I00.

\bibliography{bibliography}
\bibliographystyle{/home/guille/math/bmosparse/alpha-sort}
\end{document}